\documentclass[reqno,12pt]{amsart}

\usepackage{amsmath}
\usepackage{amsthm}
\usepackage{amssymb}
\usepackage{euscript}
\usepackage{mathrsfs}

\usepackage{float}  
\usepackage{array}  

\newtheorem{Lemma}[equation]{Lemma}

\newtheorem{Corollary}[equation]{Corollary}
\newtheorem{Proposition}[equation]{Proposition}

\theoremstyle{remark}

\theoremstyle{remark}

\theoremstyle{definition}

\numberwithin{equation}{section}

\newcommand{\R}{\mathrm{\bf R}}           
\newcommand{\C}{\mathrm{\bf C}}           
\newcommand{\Ad}{\mathrm{Ad}}             

\newcommand{\rank}{\mathrm{rank}}
\newcommand{\sym}{\mathrm{sym}}

\newcommand{\supp}{\mathrm{supp}}
\newcommand{\ann}{\mathrm{ann}}
\newcommand{\Irr}{\mathrm{Irr}}

\newcommand{\fa}{{\mathfrak a}}             
\newcommand{\fb}{{\mathfrak b}}
\newcommand{\fB}{{\mathfrak B}}
\newcommand{\fg}{{\mathfrak g}}
\newcommand{\fh}{{\mathfrak h}}
\newcommand{\fk}{{\mathfrak k}}
\newcommand{\fl}{{\mathfrak l}}

\newcommand{\fn}{{\mathfrak n}}
\newcommand{\fp}{{\mathfrak p}}
\newcommand{\fq}{{\mathfrak q}}

\newcommand{\fu}{{\mathfrak u}}

\newcommand{\f}{\mathfrak}

\newcommand{\ga}{\alpha}
\newcommand{\gb}{\beta}

\newcommand{\gD}{\Delta}
\renewcommand{\ge}{\epsilon}

\newcommand{\gl}{\lambda}
\newcommand{\gL}{\Lambda}

\newcommand{\gs}{\sigma}

\newcommand{\gt}{\theta}

\newcommand{\Cal}{\mathcal}
\newcommand{\Eul}{\EuScript}
\newcommand{\fQ}{{\Eul Q}}
\newcommand{\fC}{{\Eul C}}
\newcommand{\fD}{{\Eul D}}
\newcommand{\fR}{{\mathcal R}}
\renewcommand{\fB}{{\mathfrak B}}
\newcommand{\fF}{{\mathfrak F}}

\newcommand{\IP}[2]{\langle#1 , #2\rangle}     

\newcommand{\Htop}{H_{\mathrm{top}}(\fB^f)}

\renewcommand{\bar}[1]{\overline{#1}}

\newcommand{\cb}{T_\fQ^*\fB}

\begin{document}

\title{Examples of leading term cycles of Harish-Chandra modules}


\author{L.~Barchini}
\address{Oklahoma State University\\
Mathematics Department\\
   Stillwater, OK 74078}
\email{ leticia@math.okstate.edu}

\author{R.~Zierau}
\address{Oklahoma State University\\
Mathematics Department\\
   Stillwater, OK 74078}
\email{ zierau@math.okstate.edu}

\thanks{Research supported by an NSA grant}

\subjclass[2010]{22E46}   

\keywords{Harish-Chandra module, symplectic group, characteristic cycle}

\begin{abstract}
Many examples are given of irreducible Harish-Chandra modules for type $C_n$ having reducible leading term cycles.  Examples are also given of irreducibles in category $\mathcal O$ having reducible associated varieties in type $C_n$.
\end{abstract}

\maketitle

\parskip=4pt
\baselineskip= 16pt

\section*{Introduction} 
The characteristic cycle is an important invariant of a Harish-Chandra module $X$. It is defined in terms of the localization of $X$.  The leading term cycle is the portion of the characteristic cycle that contributes to the associated variety.  In principle the leading term cycle is easier to calculate than the characteristic cycle, but there is know know method to compute either in any generality.  Until recently it was not known whether or not the leading term cycle of an irreducible Harish-Chandra module of $SU(p,q)$ is necessarily irreducible; an example given in \cite{Williamson14} for category $\Cal O$ implies that there are in fact examples of reducible leading term cycles for $SU(p,q)$.  It was suspected that leading term cycles of irreducible Harish-Chandra modules for $Sp(p,q)$ are always irreducible.  In this article we show there are many Harish-Chandra modules of indefinite symplectic groups having reducible leading term cycles.  An example of an irreducible Harish-Chandra module for $Sp(4,4)$ having reducible leading term cycle is provided in \cite{Barchini14}.  Here we give a systematic way to find many examples.  

The argument is based on a method given in \cite[\S4]{Trapa07} to transfer information about leading term cycles from one real form to another.  Briefly, we construct examples of reducible leading term cycles for $Sp(2n,\R)$, then make conclusions about reducible leading term cycles for indefinite symplectic groups.  A byproduct is that many irreducibles in category $\Cal O$ in type $C$ are located that have reducible associated varieties.  The Harish-Chandra modules we work with for $Sp(2n,\R)$ have highest weights.  This puts us in the intersection of a category of $(\fg,K)$-modules and a category of $(\fg,B)$-modules; the appendix gives a geometric lemma special to this situation.


\section{Preliminary statements}\label{sec:ps} 


\subsection{} \label{subsec:hwt}  Highest weight Harish-Chandra modules occur for a connected simple Lie group $G_0$ exactly when $G_0$ is of hermitian type (i.e., when the center of a maximal compact subgroup $K_0$ is one-dimensional).  Write the complexified Cartan decomposition as $\fg=\fk+\fp$ and fix a Cartan subalgebra $\fh$ of $\fg$ contained in $\fk$.  Let $K$ denote the complexification of $K_0$.  As $K$-representation, $\fp$ decomposes into a direct sum of two irreducible subrepresentations.  It is convenient to choose one of these subrepresentations and call it $\fp_+$ (and call the other $\fp_-$).  Let $\gD(\fp_+)$ be the set of $\fh$-weights in $\fp_+$.  We fix a positive system of roots $\gD_c^+\subset \gD^+(\fh,\fk)$ and set
\begin{equation}\label{eqn:pos}
 \gD^+:=\gD_c^+\cup \gD(\fp_+),
\end{equation}
a positive system of roots in $\gD(\fh,\fg)$.

Let $\fb$ be the Borel subalgebra of $\fg$ determined by $\gD^+$; $\fb=\fh+\fn,$ with $\fn=\sum_{\ga\in\gD^+} \fg^{(\fa)}$.  Then each highest weight Harish-Chandra module has a highest weight with respect to either $\gD^+$ or $\gD_c^+\cup\gD(\fp_-)$.

The highest weight modules with respect to $\gD^+$ are of the form $L(\gl)$, the irreducible quotient of $M(\gl):=\Cal U(\fg) \otimes_{\Cal U(\fb)}\C_\gl$, where $\gl\in\fh^*$.  Those of infinitesimal character $\rho$ may be written as 
$$
L_w:=L(-w\rho-\rho), w\in W(\fh,\fg).
$$

The highest weight module $L(\gl)$ is a Harish-Chandra module if and only if $\gl$ is $\gD_c^+$-dominant and analytically integral.  Therefore, the highest weight Harish-Chandra modules of infinitesimal character $\rho$  are precisely the the $L_w$ for which $-w\rho$ is $\gD_c^+$-dominant.  The number of these is $\#(W/W(\fh,\fk))$.

We let $\Cal N$ denote the nilpotent cone in $\fg$ and set $\Cal{N}_\gt:=\Cal{N}\cap\fp$.  The definition of associated variety immediately implies that each irreducible highest weight Harish-Chandra module has associated variety contained in $\fp_+\subset\Cal N_\gt$.  See, for example, \cite{NishiyamaOchiaiTaniguchi01}.  There are $r+1$ $K$-orbits in $\fp_+$, where $r=\text{rank}_\R(G_0)$.  We may label them so that 
$$
  \{0\}=\bar{\Cal{O}_0}\subset \bar{\Cal{O}_1}\subset \cdots \bar{\Cal{O}_r} =\fp_+.
$$
It follows that the associated variety of any highest weight Harish-Chandra module is the closure of a single orbit $\Cal{O}_k$.


\subsection{}\label{subsec:CC}  The group $K$ acts on $\fB$ with a finite number of orbits; we write $K\backslash \fB$ for the set of $K$-orbits on $\fB$.  The conormal variety for this action of $K$ is
$$
  T_K^*\fB:=\bigcup_{\fQ\in K\backslash\fB} \bar{T_\fQ^*\fB}.
$$

Suppose that $X$ is a Harish-Chandra module with corresponding $\fD$-module $\Cal{M}_X$.  Write the characteristic cycle of $\Cal{M}_X$ as 
$$
  CC(X)=\sum n_\fQ(X) [\bar{T_\fQ^*\fB}],
$$
where the sum is over various $K$-orbits in $\fB$.  See \cite[Section 1]{BorhoBrylinski85}.

The moment map $\mu:T^*\fB\to\Cal{N}$ relates the characteristic cycle to the  associated variety as follows.  One easily sees that $\mu(\bar{T_\fQ^*\fB})\subset \Cal{N}_\gt$ and $\mu^{-1}(\Cal N_\gt)=T_K^*\fB$.  By \cite{BorhoBrylinski85}
\begin{equation}\label{eqn:AV}
  AV(X)=\bigcup_{\fQ:n_\fQ(X)\neq 0}\mu(\bar{T_\fQ^*\fB}).
\end{equation}

Each $\mu(\bar{T_\fQ^*\fB})$ is the closure of some $K$-orbit $\Cal O$ in $\Cal N_\gt$.  The `Springer fiber' $\mu^{-1}(f)$ is an equidimensional subvariety of $T_K^*\fB$.    Its irreducible components are partitioned by the irreducible components of $T_K^*\fB$, i.e., by the conormal bundle closures $\bar{T_\fQ^*\fB}$.  So each component of $\mu^{-1}(f)$ lies inside a unique $\bar{T_\fQ^*\fB}$.   It is a fact that the set of irreducible components of $\mu^{-1}(f)$ that lie in a given conormal bundle is either empty or exactly one orbit under the component group
$A_K(f):=Z_K(f)/Z_K(f)_e.$
Therefore, for a given $f\in\Cal N_\gt$,
\begin{equation}\label{eqn:AK-orbit}
  \#\{\fQ: \mu(\bar{T_\fQ^*\fB})=\bar{K\cdot f}\}=
  \#\{A_K(f)\text{-orbits in }\mu^{-1}(f)\}.
\end{equation}  

The notion of Harish-Chandra cell is defined in \cite{BarbaschVogan83}.  A cell spans a representation $V_\fC$ of $W$, which is defined in terms of coherent continuation.  It follows easily from the definition that if a cell $\fC$ contains one highest weight Harish-Chandra module, then all representations in $\fC$ are highest weight Harish-Chandra modules.  It also follows easily from the definitions that any two representations in a Harish-Chandra cell have the same  associated variety.   

Suppose $\fC$ is a cell and $\Cal O\subset \Cal N_\gt$ is a $K$-orbit for which $\bar{\Cal O}$ occurs in the associated variety for $\fC$.  Write $\Cal O^\C$ for the $G$-saturation of $\Cal O$.  Then the representation $\pi(\Cal O^C)$ associated to $\Cal O^\C$  by the Springer correspondence occurs in $V_\fC$.  Through the projection $T_\fQ^*\fB \to \fB$, the Springer fiber $\mu^{-1}(f)$ is identified with $\fB^f:=\{\fb\in\fB:f\in \fb\}$ and we have (\cite{Springer78})
$$\pi(\Cal O^\C):=  \left( H^{\text{top}}(\fB^f)\right)^{A_G(f)},
$$
with
$$
  A_G(f):=Z_G(f)/Z_G(f)_e.
$$
Thus, the dimension of $\pi(\Cal O^\C)$ is the number of $A_G(f)$-orbits in $\Irr(\fB^f):=\{\text{irreducibe components in }\fB^f\}$ (or in $\mu^{-1}(f)$).  It follows that if $A_G(f)$ and $A_K(f)$ have the same orbits in $\Irr(\fB^f)$, then for $\Cal O=K\cdot f$
$$\#\{\fQ: \mu(\bar{T_\fQ^*\fB})=\bar{\Cal O}\}=\dim(\pi(\Cal O^\C)).
$$ 

We now make an observation that will be important in Section \ref{sec:sp}.  This counting argument has been used, for example, in \cite{Trapa07}.  Suppose that $\fC$ is a Harish-Chandra cell of highest weight Harish-Chandra modules with associated variety equal to $\bar{\Cal O}=\bar{K\cdot f}$.

\begin{Lemma} \label{lem:redCC}
If the $A_G(f)$ and $A_K(f)$ orbits in $\Irr(\fB^f)$ coincide and  $\#(\fC)\gneqq \dim(\pi(\Cal O^\C))$, then there are representations in $\fC$ with reducible characteristic cycle.
\end{Lemma}

\begin{proof} No two representations in a cell of highest weight Harish-Chandra modules have the same support. This is Cor.~\ref{cor:supp} of the Appendix. Thus for some representation in $\fC$, the conormal bundle of the support has moment map image of strictly smaller dimension than $\Cal O$.  In this case the conormal bundle of support occurs in the characteristic cycle along with \emph{another} conormal bundle closure which has moment image exactly $\bar{\Cal O}$.
\end{proof}

In terms of the \emph{leading term cycle}
$$
LTC(X):=\sum_{\mu(\bar{\cb})=\bar{\Cal O}} n_\fQ(X) [\bar{T_\fQ^*\fB}],
$$
the lemma tells us that if  the size of the cell is strictly greater than $\dim(\pi(\Cal O^\C))$, then there are Harish-Chandra modules $X_\fQ\in\fC$ ($\supp(X_\fQ)=\bar{\fQ}$) for which $\bar{\cb}$ does not occur in $LTC(X_\fQ)$.


\subsection{Cohomological induction}\label{subsec:coho_ind} Let $\fq=\fl+\fu$ be a $\gt$-stable parabolic subalgebra of $\fg$ and write $\bar{\fq}=\fl+\bar{\fu}$ for the opposite parabolic. We consider cohomologically induced modules $\fR_{\bar{\fq}}(Z)$ where $Z$ is in the good range.  The good range means that if $\gL_Z$ is the infinitesimal character of $Z$, then $\IP{\gL_Z+\rho(\bar{\fu})}{\gb}>0, \gb\in\gD(\bar{\fu})$.  This condition guarantees, for example, that $\fR_{\bar{\fq}}(Z)$ is irreducible and nonzero.
(See \cite[Th.~8.2]{KnappVogan95}.)

It is well-known that the characteristic cycle of $\fR_{\bar{\fq}}(Z)$ may be given in terms of the characteristic cycle of $Z$.  For the statement of this we need to associate a $K$-orbit in $\fB$ to each $K\cap L$-orbit in $\fB_L$, the flag variety for $\fl$.  This is done as follows.  Let $\fF$ be the generalized flag variety of all parabolic subalgebras of $\fg$ conjugate to $\fq$.  Let $\pi:\fB\to\fF$ be the natural projection.  Then 
\begin{equation*}
\pi^{-1}(\fq)=\left\{\fb\in\fB: \fb\subset\fq\right\}\simeq\fB_L.
\end{equation*}
The bijection is given by $\fb_L\mapsto \fb^L:=\fb_L+\fu$, for any $\fb_L\in\fB_L$.  For $\fQ_L\in K\cap L\backslash\fB_L$ write $\fQ_L=(K\cap L)\cdot\fb_L$.  Now set $\fQ:=K\cdot\fb^L$.  Thus we have a map $\fQ_L\mapsto\fQ$ from $K\cap L$-orbits in $\fB_L$ to $K$-orbits in $\fB$.  Now we may state the formula for the characteristic cycle of $\fR_{\bar{\fq}}(Z)$.
\begin{Proposition}\label{prop:CCRQ}
If $CC(Z)=\displaystyle\sum_i n_i[\bar{T_{\fQ_{L,i}}^*\fB_L}]$, then 
\begin{equation*}
 CC(\fR_{\bar{\fq}}(Z))=\sum_in_i[\bar{T_{\fQ_{i}}^*\fB_L}],
 \end{equation*}
where $\fQ_{L,i}\mapsto \fQ_i$ as described above.
\end{Proposition}

By observing that $$\mu(\bar{\cb})=\bar{K\cdot(\fn^L\cap\fp)}=\bar{K\cdot(\fn^L\cap\fp+\fu\cap\fp)}=\bar{K\cdot(\mu(\bar{T_{\fQ_L}^*\fB_L})+\fu\cap\fp)},$$ we have the following (well-known) corollary.
\begin{Corollary}\label{cor:AVRQ}
$\displaystyle AV(\fR_{\bar{\fq}}(Z))=\bar{K\cdot(AV(Z)+\fu\cap\fp)}$.
\end{Corollary}


\section{$Sp(2n,\R)$}\label{sec:sp} We now consider $G_0=Sp(2n,\R)$.  Then $K$ is the complex group $GL(n)$.  Then, with the usual meaning of $\ge_j\in \fh^*$, the positive system 
$$  \gD^+(\fh,\fg)=\{\ge_i\pm\ge_j:i<j\}\cup\{2\ge_i\}
$$
is as in (\ref{eqn:pos}).


\subsection{}  The action of $K$ on $\fp_+\simeq \sym(n,\C)$ is by $k\cdot X=k^tXk$.  The orbits are
$$
 \Cal O_k=\{X\in \sym(n,\C) : \rank(X)=k\}, \text{ for }k=0,1,2,\dots, n.
$$
Base points may be chosen so that $f_k:=E_{11}+\dots +E_{kk}$ (where the matrix $E_{ii}$ has $1$ in the $(i,i)$ position and $0$'s elsewhere).  Each $f_k$ corresponds to the sum of  root vectors for $2\ge_1,\dots,2\ge_k$.  Thus
$$ \Cal O_k=K\cdot f_k=K\cdot\begin{pmatrix} I_k & 0 \\ 0 & 0 \end{pmatrix} \subset \sym(n,\C).
$$
Let $\Cal O_k^\C$ be the nilpotent $G$ orbit of $f_k$.

We collect some information on these orbits.  The partition associated to $\Cal O_k^\C$ has $k$ rows of length two and $2n-2k$ rows of length one.  The representation $\pi(\Cal O_k^\C)$ of $W$ associated to $\Cal O_k^\C$ by the Springer correspondence is an action of $W$ on  $(\Htop)^{A_G(f)}$ (\cite{Springer76}).  By (for example) \cite[\S10.1]{CollingwoodMcGovern93} we have 
\begin{equation}\label{eqn:dim-sp}  \dim(\pi(\Cal O_k^\C))={n\choose[\frac{k}{2}]}.
\end{equation}
As $\Htop$ is spanned by the irreducible components of $\fB^f$, we see that $\dim(\pi(\Cal O_k^\C))$ is the number of $A_G(f)$-orbits in $\Irr(\fB^f)$.  The orbits $\Cal O_k^\C$ are `special' exactly when $k$ is either even or equal to $n$.  It is known that for the pair $(G,K)=(Sp(2n),GL(n))$ the orbits of $A_G(f)$ and $A_K(f)$ in $\Irr(\fB^f)$ coincide (for any $f\in\Cal N_\gt$).  We may conclude from our earlier discussion that
\begin{equation}\label{eqn:geo} \#\{\fQ\in K\backslash \fB : \mu(\bar{\cb})=\bar{\Cal O_k}\}={n\choose [\frac{k}{2}]}.
\end{equation}
 A complex nilpotent orbit is called \emph{special} if it is the associated variety of a primitive ideal (of infinitesimal character $\rho$). Thus, the only associated varieties of Harish-Chandra modules are real forms of special orbits. 

Now consider only even $n$.   Since $\Cal O_{n-1}^\C$ is not special, there are no irreducible Harish-Chandra modules of infinitesimal character $\rho$ with associated variety $\bar{\Cal O_{n-1}}$.  Therefore the irreducible highest weight Harish-Chandra modules having supports $\fQ$ with
$$  \mu(\bar{\cb})=\bar{\Cal O_{n-1}} \text{ or }\bar{\Cal O_{n}}
$$
all have associated variety equal to $\bar{\Cal O_n}=\fp_+$.  In particular, there are ${n\choose \frac{n}{2}-1}$ irreducible highest weight Harish-Chandra modules (of infinitesimal character $\rho$) for which the closure of the conormal bundle of the support is \emph{not} in the leading term cycle.  Let $X_Q$ be such an irreducible highest weight Harish-Chandra module (with $\fQ$ its support).   We may conclude from Lemma \ref{lem:redCC} that there is a $\fQ_0$ so that 
\begin{equation}\label{eqn:CC-down}
  CC(X_\fQ)=[\bar{\cb}]+m_{\fQ_0}[\bar{T_{\fQ_0}^*\fB}]+\dots,
\end{equation}
with $m_{\fQ_0}\neq 0$ and $\mu(\bar{T_{\fQ_0}^*\fB})=\fp_+$.


\subsection{}\label{subsec:q}  For any $m=0,1,2,\dots,n-1$, a $\gt$-stable parabolic subalgebra $\fq$ is defined by 
$$  H_n:=(n-m,\dots,2,1,0,\dots,0)=\sum_{j=1}^{n-m}(n-m-j+1)\ge_j\in\fh^*.
$$
The Levi subalgebra is $\fl\simeq \C^{n-m}\oplus \f{sp}(2m)$.  Note that $\gD(\fu)\subset \gD^+(\fh,\fg)$ and it follows that 
\begin{equation}\label{eqn:u}
  \fu\cap\fp=\fu\cap\fp_+.
\end{equation}

Consider an irreducible Harish-Chandra module $X$ for $L=Sp(2m,\R)$ as described in (\ref{eqn:CC-down}), that is, the conormal bundle of support does not contribute to the leading term cycle.  Since we are now working in the group $L$, we will use a subscript $L$ for the orbits in $\fB_L$ and set $Z_{\fQ_L}:=\C\otimes X$.  The support is some $\bar{\fQ}_L\subset \fB_L$ and $CC(Z_{\fQ_L})=[\bar{T_{\fQ_L}^*\fB_L}]+m_{\fQ_{L,0}}[\bar{T_{\fQ_{L,0}}^*\fB_L}]+\dots$.  Then the infinitesimal character of $Z_{\fQ_L}$ is $\rho(\fl)$, so $Z_{\fQ_L}$ is in the good range.  By \S\ref{subsec:q}
$$  CC(\fR_{\bar{\fq}}(Z_{\fQ_L}))=[\bar{T_{\fQ}^*\fB}]+m_0[\bar{T_{\fQ_0}^*\fB}]+\dots,
$$ 
for some $K$-orbits $\fQ,\fQ_0$ in $\fB$, and 
\begin{align*}
    AV(\fR_{\bar{\fq}}(Z_{\fQ_L}))&=\bar{K\cdot(AV(Z_{\fQ_L})+\fu\cap\fp)} \\
                     &=\bar{K\cdot(\fp_+\cap\fl+\fu\cap\fp)}\\
                     &=\fp_+, \text{ by (\ref{eqn:u})}.
\end{align*}

\begin{Lemma} Both $[\bar{T_{\fQ}^*\fB}]$ and $[\bar{T_{\fQ_0}^*\fB}]$ occur in $LTC(\Cal R_{\bar{\fq}}(Z_{\fQ_L}))$.
\end{Lemma}

\begin{proof} 
We check that $\mu(\bar{T_{\fQ}^*\fB})=\mu(\bar{T_{\fQ_0}^*\fB})=\fp_+$.  As noted in \S2, under the correspondence $\fQ_L\to\fQ$
$$  \mu(\bar{T_{\fQ}^*\fB})=\bar{K\cdot(\mu(\bar{T_{\fQ_L}^*\fB})+\fu\cap\fp)},
$$
and similarly for $\fQ_{L,0}\to\fQ_0$.  Thus,
$$  \mu(\bar{T_{\fQ_0}^*\fB})=\bar{K\cdot(\fp_+\cap\fl+\fu\cap\fp)},
$$
By (\ref{eqn:u}).  For $\fQ_{L}$ we have
$$   \mu(\bar{T_{\fQ}^*\fb})=\bar{K\cdot(\bar{\Cal O_{n-1}}+\fu\cap\fp_+)}.
$$
So we need to check that $\Cal O_{n-1}+\fu\cap\fp_+$ contains matrices of rank $n$ (on identifying $\fp_+$ with $\sym(n,\C)$).  In terms of root vectors such a matrix is 
$$ \sum_{i=n-m+1}^{N-1} X_{2\ge_i} +\left( \sum_{j=2}^{n-m} X_{2\ge_j} + (X_{\ge_1+\ge_N})\right)\in f_{n-1}+\fu\cap\fp_+.
$$
\vskip-30pt
\end{proof}

\begin{Proposition} Fix $m$ even and $m<n$.  There are ${m\choose \frac{m}{2}-1}$ highest weight Harish-Chandra modules $Z$ of $Sp(2n,\R)$ for which $\Cal R_{\bar{\fq}}(Z)$ has associated variety $\fp_+$ and has reducible leading term cycles.
\end{Proposition}
\bigskip\medskip


\subsection{}\label{subsec:refined} This section contains a refined version of the previous section, giving a larger number of examples of Harish-Chandra modules with reducible leading term cycle, however we rely on a deeper result.  It is proved in \cite[Thm.~7]{McGovern98} that each Harish-Chandra cell for $Sp(2n,\R)$ is `Lusztig', that is, it is isomorphic to a corresponding left cell representation.  The dimensions of these may be calculated from \cite[Ch.~11]{Carter85}.  This is an easy calculation for the orbits $\Cal O_{2j}^\C$ and one gets the following.

\hspace{2pt}\noindent (a) When $n$ is even,  for each $j=0,1\dots,\frac{n}{2}$, there is a cell $\fC_{2j}$ having associated variety $\Cal O_{2j}$ of size ${n\choose j}+{n\choose j-1}.$

\hspace{2pt}\noindent (b) When $n$ is odd,  for each $j=0,1\dots,\frac{n-1}{2}$, there is a cell $\fC_{2j}$ having associated variety $\Cal O_{2j}$ of size ${n\choose j}+{n\choose j-1}$ and a cell $\fC_n$ with associated variety $\bar{\Cal O_n}=\fp_+$ of size ${n\choose \frac{n-1}{2}}$.

\noindent From this, along with (\ref{eqn:geo}), we may conclude the following. 

\noindent \hspace{2pt}(a) When $n$ is even, 
\begin{align*}
&\fC_{2j}=\left\{X_\fQ: \mu(\bar{\cb})=\bar{\Cal O}_{2j}\text{ or }\bar{\Cal O}_{2j-1}\right\}, j=0,1,\dots,\frac{n}{2}.  
\intertext{\hspace{2pt}(b) When $n$ is odd, }
&\fC_{2j}=\left\{X_\fQ: \mu(\bar{\cb})=\bar{\Cal O}_{2j}\text{ or }\bar{\Cal O}_{2j-1}\right\}, j=0,1,\dots,\frac{n-1}{2}
\text{ and }  \\
&\fC_{n}=\left\{X_\fQ: \mu(\bar{\cb})=\bar{\Cal O}_{n}\right\}.
\end{align*}

\begin{Proposition} With $\fq=\fl+\fu$  as in \S\ref{subsec:q}
$$\bar{K\cdot(\bar{\Cal O}_{2j-1}+\fu\cap\fp)}=\bar{K\cdot(\bar{\Cal O}_{2j}+\fu\cap\fp)}=\fp_+,$$
when $j\geq m-[\frac{n-1}{2}]$.
\end{Proposition}

\begin{proof}  Let $0\leq k\leq m$. We have
\begin{align*}
\bar{K\cdot(\bar{\Cal O}_{k}+\fu\cap\fp)}& \\
&\hspace{-62pt} =\bar{K\cdot(\bar{\Cal O}_{k}+\fu\cap\fp_+)}, \text{ since }\fu\cap\fp=\fu\cap\fp_+, \\
&\hspace{-62pt} =\bar{K\cdot(\Cal O_{k}+\fu\cap\fp_+)},  \\
&\hspace{-62pt} =\bar{K\cdot(f_{k}+\fu\cap\fp_+)}, \text{ since }K\cap L\text{ preserves }\fu\cap\fp_+.
\end{align*}
\noindent \underline{\bf Claim:} \quad
$\displaystyle{\bar{K\cdot(f_{k}+\fu\cap\fp)}=
\begin{cases} \bar{\Cal O}_{2(n-m)+k},  & k=0,1,\dots ,2m-n-1  \\
              \bar{\Cal O}_{n},  & k\geq 2m-n.
\end{cases}}$

\noindent Case 1. $k\geq 2m-n$  (i.e., $m-k\leq n-m$).  Then
$$f_k+\sum_{i=1}^{m-k}X_{\ge_i+\ge_{n-m+k+i}} + \sum_{i=m-k+1}^{n-m} X_{2\ge_i}\in f_k+\fu\cap\fp_+
$$
has rank $n$.  Note that in $\sym(n,\C)$ this is the matrix
$$\begin{bmatrix}
  &  &  &  I_{m-k}  \\
  &  I_{(n-m)-(m-k)}&  &    \\
  &  & I_k &    \\
 I_{m-k}   &  &  &    \\
\end{bmatrix}.
$$

\noindent Case 1. $k\leq 2m-n$ (i.e., $m-k\geq n-m$).  In this case $f_k+\fu\cap\fp_+$ contains
$$ f_k+\sum_{i=1}^{n-m} X_{\ge_i+\ge_{n-m+k+i}},
$$
which has rank $k+2(n-m)$. In $\sym(n,\C)$ this is the matrix
$$\begin{bmatrix}
  &  &    I_{n-m} & \\
  &  I_k&  &    \\
  I_{n-m}  &  & &    \\
    &  &  &  0_{m-k}  \\
\end{bmatrix}.
$$
Since the lower right $m\times m$ block as rank $k$, the greatest possible rank in $f_k+\fu\cap\fp_+$ is $(n-m)+k+(n-m)$.
\end{proof}

We may conclude the following.
\begin{Proposition}\label{prop:conclusion}  Let $j_0=\max\{0,m-[\frac{n-1}{2}]\}$ and $j=j_0,j_0+1,\dots,[\frac{m}{2}]$.  Then there are ${ m\choose j-1}$ highest weight Harish-Chandra modules $Z_{\fQ_L}$ for $(\fl,K\cap L)$ with associated variety $\bar{\Cal O_{2j}}$ for which $X_{\fQ}=\Cal R_{\bar{\fq}}(Z_{\fQ_L})$ has associated variety $\fp_+$ and has reducible leading term cycle.  
\end{Proposition}


\subsection{} There is a relationship between leading term cycles of irreducible Harish-Chandra modules and associated varieties of irreducible highest weight modules.  This is discussed in detail in \cite[Section 4]{Trapa07}.  We summarize the conclusions that we will use.  We may consider arbitrary Harish-Chandra modules for arbitrary reductive algebraic groups for this discussion.

A well-known geometric fact is that there is a one-to-one correspondence 
$$ \left\{ A_G(f)\text{-orbits in }\Irr(\fB^f)\right\}\leftrightarrow \left\{\text{orbital varieties in }\fn\cap(G\cdot f)\right\}.
$$

Let $X_\fQ$ be an irreducible Harish-Chandra module with support  $\bar{\fQ}$.  Suppose that $AV(X_\fQ)=\bar{\Cal O_K}$ and $\ann(X_\fQ)=\ann(L_w)$ for some $w\in W$.  We note that if, for $f\in \Cal N_\gt$, the $A_G(f)$ and $A_K(f)$ orbits in $\Irr(\fB^f)$ coincide, then there is a one-to-one correspondence
\begin{equation}\label{eqn:corr}
 \left\{\fQ\in K\backslash\fB : \mu(\bar{\cb})=\bar{\Cal O_K}\right\} \leftrightarrow \left\{\text{orbital varieties in }\fn\cap(\Cal O_K^{\,\C})\right\}.
\end{equation}
The following two statements are contained in \cite[Cor.~4.2]{Trapa07}
\begin{Proposition}\label{prop:tr} Suppose the $A_G(f)$ and $A_K(f)$ orbits  in $\Irr(\fB^f)$ coincide.  If $\fQ\leftrightarrow \nu$ under the one-to-one correspondence (\ref{eqn:corr}), then $[\bar{\cb}]$ occurs in $LTC(X_\fQ)$ if and only if $\nu$ occurs as a component of $AV(L_{w^{-1}})$.
\end{Proposition}
Note that without the equal orbit hypothesis we might have a reducible leading term cycle with, say, $\bar{T_{\fQ_1}^*\fB}$ and $\bar{T_{\fQ_2}^*\fB}$ occurring, for which the corresponding $A_K(f)$-orbits in $\Irr(\fB^f)$ lie in the same $A_G(f)$-orbit.  In this case $AV(L_{w^{-1}})$ may still be irreducible.  However, without the equal orbits hypothesis we have the following.
\begin{Proposition}\label{prop:2}
Suppose $\ann(X_\fQ)=\ann(L_w)$.  If $AV(L_{w^{-1}})$ is reducible, then $LTC(X_\fQ)$ must contain at least two conormal bundle closures.
\end{Proposition}

Prop.~\ref{prop:tr}, along with \ref{subsec:refined}, gives a family of irreducible highest weight modules $L_{w^{-1}}$ of $\f{sp}(2n)$ having reducible associated varieties.  The $w$'s can be determined explicitly as follows.  Begin with $\gs\in W_\fl$, $\fl=\f{sp}(2m)$, so that $\bar{B(\fn\cap\fn^\gs)}=\bar{\Cal O}_{2j-1}$ with $j$ as in Prop.~\ref{prop:conclusion}.  Then $\Cal R_{\bar{\fq}}(L_\gs)$ is the highest weight Harish-Chandra module $L_w$, $w=w^0\gs$, where $w^0$ is the product of the long elements of $W_\fk$ and $W_{\fk\cap\fl}$ (as follows easily from the the formula for the lowest $K$-type of a cohomologically induced representation in the good range (\cite{VoganZuckerman84})).  Thus the annihilator of $\Cal R_{\bar{\fq}}(L_\gs)$ is $ann(L_w)$ so Prop.~\ref{prop:tr} tells us that $AV(L_{w^{-1}})$ is reducible.


\subsection{Reducible leading term cycles for $Sp(n,n)$}
Now consider $G_0'=Sp(p,q)$, $p+q=n,$ so $G'=Sp(2n)$ and $K'\simeq Sp(2p)\times Sp(2q)$.  Suppose that $\Cal O$ is a $K'$-orbit in $\Cal N_\gt'$.  Then 
\begin{enumerate}
\item[(i)] there is a Harish-Chandra cell with associated variety $\bar{\Cal O}$ (\cite{McGovern98}), and 
\item[(ii)] if $I$ is a primitive ideal with $AV(I)=\bar{\Cal O^C}$, then there is a Harish-Chandra module $X$ with $\ann(X)=I$.
\end{enumerate}

There are irreducible  representations of $Sp(2n,\R)$ with reducible leading term cycles and associated variety $\bar{\Cal O}_n=\fp_+$ (\S\ref{subsec:q},\ref{subsec:refined}).  In particular, these representations each have an annihilator $I$ with $AV(I)=\bar{\Cal O_n^\C}$.  But $\bar{\Cal O_n^\C}$ corresponds to the partition $(2,2,\dots,2)$.  It follows\footnote{The $K'$-orbits occurring in $\Cal N_\gt'$ are those for which 
(i) there are  an even number of even rows and the same number beginning with $+$ as those beginning with a $-$, and
(ii) there are an even number of odd rows and an even number beginning with $+$ and an even number beginning with a $-$.} that $\Cal O_n^C$ is $\Cal O^\C$ for some $K'$-orbit $\Cal O$ in $\Cal N_\gt'$ with $p=q=\frac{n}{2}$ (and $n$ must be even).

See \cite{Barchini14} for an  explicit example in type $C_4$.

We may now conclude that there are irreducible representations of $Sp(\frac{n}{2},\frac{n}{2})$ with reducible leading term cycles.


\appendix

\section{Geometric facts}  Suppose $X$ is a highest weight Harish-Chandra module.  As we know, the associated variety of $X$ considered as a highest weight module, is a union of components of $\fn\cap\Cal O^\C$ (where $\bar{\Cal O^\C}$ is the associated variety of the annihilator of $X$).  On the other hand, the associated variety of a Harish-Chandra module is a union of closures of $K$-orbits in $\Cal N_\gt\cap \Cal O^\C$.  These two seemingly different objects coincide: each may be computed using the good filtration $\{X_n\}$, where $X_0$ is the $K$-type generated by the highest weight vector and $X_n=\Cal U_n(\fp^-)\cdot X_0$.  This filtration is stable under \emph{both} $K$ and $B$.  Lemma \ref{lem:ov} below shows that the the relevant orbital varieties and nilpotent $K$ orbits are in fact the same.

Similarly, the characteristic cycle of a highest weight module is a union of closures of conormal bundles of $B$-orbits on $\fB$ and the the characteristic cycle of a Harish-Chandra module is the union of closures of $K$-orbits on $\fB$.  These also coincide: the localization of $X$ is $\Eul X:=\Cal D\underset{\Cal U(\fg)}{\otimes}X$ and the  characteristic cycle is defined by any good filtration, thus independent of whether we consider $X$ a highest weight module or a Harish-Chandra module. Lemma \ref{lem:schubert} shows that the relevant Schubert varieties and $K$-orbit closures coincide.  A statement in the same spirit is in \cite[Section 1]{Collingwood92}.

Suppose $X=L_w$.  As $X$ is a Harish-Chandra module, $-w\rho$ is $\gD_c^+$-dominant.


\subsection{Orbital varieties}   Suppose $y\in W$,  then $\nu(y):=\bar{B\cdot (\fn\cap\fn^y)}$, with $\fn^y=\Ad(y)\fn$, is an irreducible component of $\bar{\fn\cap\Cal O^\C}$.  Each such component has the form $\nu(y)$ for some $y\in W$ and is called an \emph{orbital variety}.  One easily sees that 
\begin{equation}\label{eqn:1}
\nu(y)\subset \fp_+\text{ if and only if $-y\rho$ is $\gD_c^+$-dominant.}
\end{equation}

\begin{Lemma}\label{lem:ov}
{\upshape{(i)}} If $\Cal O_K:=K\cdot f\subset\fp_+$, then $\bar{\Cal O_K}=\nu(y)$ for some $y\in W$.  In this case  $-y\rho$ is necessarily $\gD_c^+$-dominant.  \newline
{\upshape (ii)} If $\nu(y)\subset \fp_+$, then $\nu(y)=\bar{\Cal O_K}$ for some $K$-orbit in $\fp_+$.
\end{Lemma}
\begin{proof} (i) Since $BK\cdot f=KP_+\cdot f=K\cdot f$ (as $P_+$ is abelian and $f\in \fp_+$), $\Cal O_K$ os $B$-stable.  Also, $\dim(\Cal O_K)=\frac12\dim(\Cal O_K^\C)$ is the dimension of any orbital variety for $\Cal O^\C$.  It follows that $\bar{\Cal O_K}$ is $\bar{\nu(y)}$ for some $y\in W$.  By (\ref{eqn:1}), $-y\rho$ is $\gD_c^+$-dominant.

(ii) If $\nu(y)\subset \fp_+$, then 
\begin{align*}
K\nu(y)& =K\bar{B(\fn\cap\fn^y)}  \\
       & \subset \bar{KB(\fn\cap\fn^y)}  \\
       & \subset \bar{KP_+(\fn\cap\fn^y)}  \\
       & =\bar{K(\fn\cap\fn^y)},\text{ since }\fn\cap\fn^y\subset\fp_+, \\
       & \subset \bar{\bar{(K\cap B)(K\cap B^-)}(\fn\cap\fn^y)}, \\
       &  \hskip 36pt     \text{ since $(K\cap B)(K\cap B^-)$ is dense in }K,  \\
       &\subset  \bar{(K\cap B)(K\cap B^-)(\fn\cap\fn^y)}, \\
       &\subset \bar{B (\fn\cap\fn^y)}  \\
       &= \nu(y).
\end{align*}
The last inclusion needs justification.  If suffices to show that $\ga\in\gD_c^+$ and $\gb\in\gD(\fn\cap\fn^y)$ implies $-\ga+\gb\in \fn\cap\fn^y$.  Since $\gb\in\gD(\fp_+)\subset \fn$, we need only show $-\ga+\gb\in\fn^y$, i.e., $y^{-1}(-\ga+\gb)>0$.  However, $\IP{-\ga+\gb}{y\rho}=\IP{\ga}{-y\rho}+\IP{y^{-1}\gb}{\rho}>0$, since $-y\rho$ is $\gD_c^+$-dominant by (\ref{eqn:1}) and $y^{-1}\gb>0$ (since $\gb\in\fn^y$).
\end{proof}

\subsection{} Let us set $B(y)=By\cdot\fb$ and $X(y):=\bar{By\cdot\fb}$, the Schubert variety for $y\in W$.

\begin{Lemma}\label{lem:schubert}
{\upshape (i)} If $\fQ\in K\backslash \fB$ is $B$-stable, then $\bar{\fQ}=X(y)$ or some $y\in W$ with $-y\rho$ $\gD_c^+$-dominant.  \newline
{\upshape (ii)} If $X(y)$ is $K$-stable (so the closure of a $K$-orbit in $\fB$), then $-y\rho$ is $\gD_c^+$-dominant.
\end{Lemma}

\begin{proof} (i) A closed $B$-stable irreducible subvariety of $\fB$ is a Schubert variety, so $\bar{\fQ}=X(y)$ for some $y\in W$.  We need to show $-y\rho$ is $\gD_c^+$-dominant.  If not, there is $\ga\in\gD_c^+$ so that $\IP{y\rho}{\ga}>0$, i.e., $y^{-1}\ga>0$.  Therefore, $\ell(y)<\ell(\gs_\ga y)$, where $\gs_\ga$ is the simple reflection for $\ga$.  Since $KBy\cdot \fb\subset X(y)$, $\gs_\ga y\cdot \fb\in X(y)$.  This implies that $X(\gs_\ga y) \subset X(y)$.  But this is impossible since $\ell(y)<\ell(\gs_\ga y)$.

(ii) Since $X(y)$ is $K$-stable, it is the closure of a single $K$-orbit in $\fB$.  The argument above shows that $-y\rho$ is $\gD_c^+$-dominant.
\end{proof}


\subsection{}  We may draw a few conclusions.
Suppose $X=L_w$ is a Harish-Chandra module, so $-w\rho$ is $\gD_c^+$-dominant.  The support of the localization is $X(w)$ (\cite[Prop.~6.4]{BrylinskiKashiwara81}).  By Lemma \ref{lem:schubert} this is the closure of a $K$-orbit $\fQ_w$ in $\fB$.  

\begin{Lemma} With the above notation $\bar{T_{B(w)}^*\fB}=\bar{T_{\fQ_w}^*\fB}$.
\end{Lemma}

\begin{proof} The varieties $X(w)$ and $\bar{\fQ_w}$ contain a smooth dense subvariety in common (namely $(Bw\cdot\fb)\cap\fQ_w$).   Both conormal bundles appearing in the lemma are the closure of the conormal bundle to the common dense subvariety.
\end{proof}

\begin{Corollary}\label{cor:CC}  The characteristic cycle of $X$ is a combination of conormal bundle closures $\bar{T_{\fQ_y}^*\fB}$ with $-y\rho $ $\gD_c^+$-dominant.
\end{Corollary}
\begin{proof}  Each conormal bundle occurring is the conormal bundle of  both a $K$-orbit and  $B$-orbit, so Lemma \ref{lem:schubert} applies. 
\end{proof}

We also conclude the following fact.  
\begin{Corollary}\label{cor:supp}
No two highest weight Harish-Chandra modules have the same support.
\end{Corollary}


\providecommand{\bysame}{\leavevmode\hbox to3em{\hrulefill}\thinspace}
\providecommand{\MR}{\relax\ifhmode\unskip\space\fi MR }
\providecommand{\MRhref}[2]{%
  \href{http://www.ams.org/mathscinet-getitem?mr=#1}{#2}
}
\providecommand{\href}[2]{#2}


\begin{thebibliography}{10}

\bibitem{BarbaschVogan83}
Barbasch.~D, and D.~Vogan, \emph{Weyl group representations and nilpotent
  orbits}, Representation theory of reductive groups (Park City, Utah, 1982)
  (Boston) (P.~C. Trombi, ed.), Progr. Math., vol.~40, BirkhŠuser, 1983,
  pp.~21--33.

\bibitem{Barchini14}
Barchini, L, \emph{Two triangularity results and invariants of
  $(\mathfrak{sp}(p+q, {C}), {S}p(p) \times {S}p(q))$ and $(\mathfrak{so}(2n,
  \text{{C}}), {GL}(n))$ modules}, submitted.

\bibitem{BorhoBrylinski85}
Borho, W., and J.-L. Brylinski, \emph{Differential operators on homogeneous
  spaces. {III}}, Invent.~Math. \textbf{80} (1985), 1--68.

\bibitem{BrylinskiKashiwara81}
Brylinski, J.-L.,  and M.~Kashiwara, \emph{Kazhdan-lusztig conjecture and holonomic
  systems}, Invent.~Math. \textbf{64} (1981), 387--410.

\bibitem{Carter85}
Carter, R, { Finite groups of {L}ie {T}ype. {C}onjugacy classes and complex characters}, A Wiley-Interscience, 1985.

\bibitem{CollingwoodMcGovern93}
Collingwood, D., and W.~McGovern, \emph{Nilpotent {O}rbits in {S}emisimple {L}ie
  {A}lgebras}, Van Nostrand Reinhold Co., New York, 1993.

\bibitem{Collingwood92}
Collinwood, D., \emph{Orbits and characters associated to highest weight
  representations}, Proc.~Amer.~Math.~Soc. \textbf{114} (1992), no.~4,
  1157--1165.

\bibitem{KnappVogan95}
Knapp, A.~W., and D.~A. Vogan, \emph{Cohomological {I}nduction and {U}nitary
  {R}epresentations}, Princeton Mathematical Series, vol.~45, Princeton
  University Press, 1995.

\bibitem{McGovern98}
McGovern, W., \emph{Cells of {H}arish-{C}handra modules for real classical
  groups}, Amer.~J.~Math. \textbf{120} (1998), 211--228.

\bibitem{NishiyamaOchiaiTaniguchi01}
Nishiyama, K., H.~Ochiai, and K.~Taniguchi, \emph{Bernstein degree and
  associated cycles of {H}arish-{C}handra modules - {H}ermitian case},
  Nilpotent Orbits, Assoicated Cycles and Whittaker Models for Highest Wieght
  Representations (K.~Nishiyama, H.~Ochiai, K.~Taniguchi, H.~Yamashita, and
  S.~Kato, eds.), Ast\'erique, vol. 273, Soc. Math. France, 2001, pp.~13--80.

\bibitem{Springer76}
Springer, T.~A., \emph{Trigonometric sums, {G}reen functions of finite groups
  and representations of {W}eyl groups}, Invent.~Math. \textbf{36} (1976),
  173--207.

\bibitem{Springer78}
\bysame, \emph{A construction of representations of {W}eyl groups.},
  Invent.~Math. \textbf{44} (1978), 279--293.

\bibitem{Trapa07}
Trapa, P., \emph{Leading-term cycles of {H}arish-{C}handra modules and partial
  orders on components of the {S}pringer fiber}, Compos. Math. \textbf{143}
  (2007), no.~2, 515--540.

\bibitem{VoganZuckerman84}
Vogan, D.A., and G.~D. Zuckerman, \emph{Unitary representations with non-zero
  cohomology}, Compos.~Math. \textbf{53} (1984), 51--90.

\bibitem{Williamson14}
Williamson, G., \emph{A reducible characteristic variety in type {A}},
  arXiv:1405.3479.

\end{thebibliography}
\end{document}